%% file: kase-4th-1.tex
\documentclass[a4,12pt]{amsart}

\usepackage{amssymb}
\usepackage{amstext}
\usepackage{amsmath}
\usepackage{amscd}
\usepackage{amsthm}
\usepackage{amsfonts}
\usepackage{enumerate}
\usepackage{graphicx}
\usepackage{latexsym}
\usepackage[all]{xy}
\usepackage[usenames]{color}

\newtheorem*{corollary*}{Corollary}
\newtheorem{theorem}{Theorem}[section]
\newtheorem{corollary}[theorem]{Corollary}
\newtheorem{lemma}[theorem]{Lemma}

\newtheorem*{proposition*}{Proposition}
\newtheorem{question}[theorem]{Question}

\newtheorem*{claim*}{Claim}

\theoremstyle{definition}
\newtheorem{definition}[theorem]{Definition}

\newtheorem{example}[theorem]{Example}
\newtheorem*{notation}{Notation}

\newtheorem{definitiontheorem}[theorem]{Definition-Theorem}

\theoremstyle{remark}

\numberwithin{equation}{theorem}
\renewcommand{\mod}{\operatorname{mod}}

\newcommand{\Hom}{\operatorname{Hom}}
\newcommand{\add}{\operatorname{\mathsf{add}}}
\newcommand{\Ext}{\operatorname{Ext}}

\newcommand{\Tau}{\mathsf{T}}

\newcommand{\tilt}{\operatorname{\mathsf{tilt}}}

\newcommand{\stilt}{\operatorname{\mathsf{s-tilt}}}
\newcommand{\sttilt}{\operatorname{\mathsf{s\tau -tilt}}}

\newcommand{\Z}{{\mathbb{Z}}}
\newcommand{\Endhom}{\operatorname{End}}
\newcommand{\dimvec}{\operatorname{\underline{dim}}}

\begin{document}
\title{Taking tilting modules from the poset of support tilting modules }
\author{Ryoichi Kase}
\address{Department of Mathematics,
Nara Women's University, Kitauoya-Nishimachi, Nara city, Nara 630-8506, Japan}
\email{r-kase@cc.nara-wu.ac.jp}
\begin{abstract}
C. Ingalls and H. Thomas defined support tilting modules for path algebras. From $\tau$-tilting theory introduced by T. Adachi, O. Iyama and I. Reiten, a partial order on the set of
basic tilting modules defined by D. Happel and L. Unger is extended as a partial order on the set of support tilting modules. 
In this paper, we study a combinatorial relationship between the poset of basic tilting modules and basic support tilting modules.
 We will show that the subposet of tilting modules is uniquely determined by the poset structure of the set of support tilting modules.
\end{abstract}
\maketitle

\section{Introduction}
Tilting theory first appeared in an article by Brenner and Butler\;\cite{BB}. In that article the notion of a tilting module for finite dimensional algebras was introduced. 
Tilting theory  now appears in many areas of mathematics, for example algebraic geometry, theory of algebraic groups and algebraic topology. Let $T$ be a tilting module 
for a finite dimensional algebra $\Lambda$ and let $B=\Endhom_{A}(T)$. Then Happel showed that the two bounded derived categories
$\mathcal{D}^{\mathrm{b}}(A)$ and $\mathcal{D}^{\mathrm{b}}(B)$
are equivalent as triangulated category\;\cite{H}.  Therefore, classifying tilting modules is an important problem.

Tilting mutation introduced by Riedtmann and Schofield is an approach to this problem. It is an operation which gives a new tilting module from given one by replacing an
 indecomposable direct summand.
  They also introduced  a tilting quiver whose vertices are 
  (isomorphism classes of) basic tilting modules and arrows correspond to mutations.  
  Happel and Unger defined a partial order on the set of basic tilting modules 
   and showed that the tilting quiver coincides with the Hasse quiver of this poset. 
   However, tilting mutation is often impossible depending on a choice of an indecomposable direct summand. Support $\tau$-tilting modules introduced by Adachi, Iyama and Reiten 
   are generalization of tilting modules. They showed that a mutation (resp.\;a partial order) on the set of (isomorphism classes of) basic tilting modules 
    is extended as an operation (resp.\;a partial order) on the set of 
   (isomorphism classes of) support $\tau$-tilting modules. They also showed that support $\tau$-tilting mutation has following nice properties:\\  
$\bullet$\;Support $\tau$-tilting mutation is always possible.\\
$\bullet$\;Support $\tau$-tilting quiver coincides with the Hasse quiver of the poset of support $\tau$-tilting modules.  

\begin{notation}
Throughout this paper, let $\Lambda$ be a finite dimensional algebra over an algebraically closed field $k$.
\begin{enumerate}
\item We always assume that $\Lambda$ is basic and indecomposable.
\item We denote by $\mod\Lambda$ 
 the category of finitely generated right $\Lambda$-modules.
\item We denote by $\tau$ the Auslander-Reiten translation of $\mod\Lambda$.
\item A module means a finitely generated right module.
\item We denote by $\tilt(\Lambda)$ (respectively, $\stilt(\Lambda)$, $\sttilt(\Lambda)$) the (partially ordered) set of (isomorphism classes of) basic tilting 
(respectively, support tilting, support $\tau$-tilting) $\Lambda$-modules (see Section\;2 below for the definition).
\end{enumerate}
\end{notation}

In this paper, we consider a combinatorial relationship between $\sttilt(\Lambda)$ and $\tilt(\Lambda)$. Since tilting or support $\tau$-tilting mutation is introduced for the aim of 
obtaining many tilting modules, the following is an interesting question.

\begin{question}
Is the set of tilting modules $\tilt(\Lambda)$ uniquely determined by the poset-structure of $\sttilt(\Lambda)$?.
\end{question} 
The main result of this paper is the following.

\begin{theorem}
\label{mainthm}
Let $\Lambda$ and $\Gamma$ be two finite dimensional basic hereditary algebras. If $\rho$
is a poset isomorphism from $\sttilt(\Lambda)$ to $\sttilt(\Gamma)$, then the restriction of $\rho$ to $\tilt(\Lambda)$ induces
a poset isomorphism
\[\rho|_{\tilt(\Lambda)}:\tilt(\Lambda)\simeq \tilt(\Gamma).\]
\end{theorem} 

\section{Preliminaries}
\subsection{Tilting modules}

In this subsection we recall the definition of tilting modules.
For a module $M$, we denote by $|M|$ the number of non-isomorphic indecomposable direct summand of $M$.

\begin{definition}
A $\Lambda$-module $M$ is said to be a \emph{partial tilting module} if it satisfies following conditions. 
\begin{enumerate}[(i)]
\item $\mathrm{pd}\;T\leq 1.$
\item $\Ext_{\Lambda}^{1}(T,T)=0.$
\end{enumerate} 

If partial tilting module $T$ satisfies $|M|=|\Lambda|$, then we call $T$ a \emph{tilting module}. The set of non-isomorphic basic tilting modules of $\Lambda$ is denoted by $\tilt\Lambda$
\end{definition}
For a $\Lambda$-module $M$, we put $M^{\perp_{1}}:=\{X\in \mod \Lambda\mid \Ext_{\Lambda}^{1}(M,X)=0\}.$
\begin{definitiontheorem}\cite{HU1}
Let $T_{1}$ and $T_{2}$ be two tilting modules. We write $T_{1}\leq T_{2}$ if
$T^{\perp_{1}}_{1}\subset T^{\perp_{1}}_{2}$. Then $\leq$ defines a partial order
on $\tilt(\Lambda)$. 
\end{definitiontheorem}
It is known that if $T$ is a tilting module, then $X$ is in $T^{\perp_{1}}$ if and only if $X$ is a factor module of finite direct sums of copies of $T$.
Therefore we have the following\;\cite{HU2}\;:
\[T\geq T^{'}\Leftrightarrow \Ext_{\Lambda}^{1}(T,T^{'})=0.\]

\subsection{Support $\tau$-tilting modules}

The notion of support $\tau$-tilting modules which were introduced in \cite{AIR}
is a generalization of that of tilting modules.

Let us recall the definition of support $\tau$-tilting modules.

\begin{definition}\cite{AIR}
Let $M$ be a $\Lambda$-module and $P$ be a projective $\Lambda$-module.
\begin{enumerate}[(1)]
\item $M$ is said to be a $\tau$-\emph{rigid module} if it satisfies
$\Hom_{\Lambda}(M, \tau M)=0$.
\item  $(M,P)$ is said to be a $\tau$-\emph{rigid pair} 
if $M$ is a $\tau$-rigid module and $\Hom_{\Lambda}(P,M)=0$.
\item $(M,P)$ is called a \emph{support $\tau$-tilting pair} 
if it is a $\tau$-rigid pair with $|M|+|P|=|\Lambda|$.
We then call $M$ a \emph{support $\tau$-tilting module}.
The set of non-isomorphic basic support $\tau$-tilting modules of $\Lambda$ is denoted by $\sttilt\Lambda$.
\end{enumerate}
\end{definition}
We note that if $M\in \sttilt(\Lambda)$, then there is a unique (up to isomorphism) basic projective module $P$ such that $(M,P)$ is a support $\tau$-tilting pair\;\cite{AIR}.
\begin{definitiontheorem}\cite{AIR}
Let $(M,P)$ and $(M^{'},P^{'})$ be two support $\tau$-tilting pair. We write $M\leq M^{'}$ if $\Hom_{\Lambda}(M,\tau M^{'})=0$ and $\add P^{'}\subset \add P $. 
Then $\leq$ defines a partial order on $\sttilt(\Lambda)$. 
\end{definitiontheorem} 
We call $(N,U)$ an almost complete support $\tau$-tilting pair if $(N,U)$ is a $\tau$-rigid pair with $|N|+|U|=|\Lambda|-1$.
\begin{theorem}\cite{AIR}
\begin{enumerate}[{\rm (1)}]
\item Let $(N,U)$ be a basic almost complete support $\tau$-tilting pair. Then $(N,U)$ is a direct summand of exactly two support $\tau$-tilting pairs.
\item Let $(M,P)$ and $(M^{'},P^{'})$ be two support $\tau$-tilting pair. Then there is an edge $M-M^{'}$ in the underlying graph of
 the Hasse quiver of $\sttilt(\Lambda)$ if and only if there exists basic almost complete support $\tau$-tilting pair $(N,U)$ such that $(N,U)$
  is a direct summand of $(M,P)$ and $(M^{'},P^{'})$.
\end{enumerate}   
\end{theorem}
For a basic $\tau$-rigid module $U$, we denote by $\sttilt_{U}(\Lambda):=\{T\in \sttilt(\Lambda)\mid U\in \add T\}$.   
\begin{theorem}\cite{J}
Let $U$ be a basic partial $\tau$-tilting module. Then there is a finite dimensional algebra $C$ with $|C|=|\Lambda|-|U|$ such that 
$\sttilt_{U}(\Lambda)\simeq \sttilt(C)$.\\
  
\end{theorem}

\subsection{Hereditary case}
Let $Q$ be a finite connected acyclic quiver. We denote by $Q_{0}$\;(resp.\;$Q_{1}$) the set of vertices\;(resp.\;arrows) of $Q$. From now on, we assume that $\Lambda$ is a path algebra $kQ$. In this paper, for any paths $w:a_{0}\stackrel{\alpha_{1}}{\rightarrow}a_{1}\stackrel{\alpha_{2}}{\rightarrow}\cdots \stackrel{\alpha_{r}}{\rightarrow} a_{r}$ and $w^{'}:b_{0}\stackrel{\beta_{1}}{\rightarrow}b_{1}\stackrel{\beta_{2}}{\rightarrow}\cdots \stackrel{\beta_{s}}{\rightarrow} b_{s}$ in $Q$, the product is defined by \[w\cdot w^{'}:=\left\{\begin{array}{ll}
 a_{0}\stackrel{\alpha_{1}}{\rightarrow}a_{1}\stackrel{\alpha_{2}}{\rightarrow}\cdots \stackrel{\alpha_{r}}{\rightarrow} a_{r}=b_{0}\stackrel{\beta_{1}}{\rightarrow}b_{1}\stackrel{\beta_{2}}{\rightarrow}\cdots \stackrel{\beta_{s}}{\rightarrow} b_{s} & \mathrm{if\ }a_{r}=b_{0} \\ 
 0 & \mathrm{if\ }a_{r}\neq b_{0},
 \end{array}\right.\]
 in $kQ$.
 For a module $M\in \mod \Lambda$, we denote by $Q(M)$ the full subquiver of $Q$ with $Q(M)_{0}=\mathrm{supp}(M):=\{a\in Q_{0}\mid (\dimvec M)_{a}>0\}$. By definition, we can regard $M$ as a sincere $kQ(M)$-module. 
\begin{definition}\cite{AIR,IT}
A $\Lambda$-module $M$ is said to be a \emph{support tilting module} if $M$ is a tilting $kQ(M)$-module.
\end{definition}
Since $\Lambda$ is a finite dimensional hereditary algebra, we have $\sttilt(\Lambda)=\stilt(\Lambda)\;(\mathrm{see}\ $\cite{AIR}) and the partial order on $\stilt(\Lambda)$ is defined as follows:
\[M\geq M^{'}\Leftrightarrow \Ext_{\Lambda}^{'}(M,M^{'})=0\ \mathrm{and}\ Q(M^{'})_{0}\subset Q(M)_{0}.\ \ (M,M^{'}\in \stilt(\Lambda))\]
\begin{theorem}\cite{J}
 Let $N$ be a basic partial tilting module. Then there is a finite dimensional hereditary algebra $C$ with $|C|=|\Lambda|-|N|$ such that 
$\stilt_{N}(\Lambda)\simeq \stilt(C)$.\\ 
\end{theorem}
Let $M\in \mod \Lambda$ and let $P$ be a projective $\Lambda$-module.
We set \[\Tau_{\Lambda}(M,P)=\Tau(M,P):=(\tau M\oplus \nu(P), M_{\mathrm{pr}} )\] where $\nu$ is the Nakayama functor and $M_{\mathrm{pr}}$ is a maximal projective direct summand. We also set \[\Tau^{-}(M,P):=(\tau^{-}M\oplus P, \nu^{-}M_{\mathrm{in}})\] where $M_{\mathrm{in}}$ is a maximal injective direct summand.
 Note that
\[\Tau\Tau^{-}(M,P)=(M,P)=\Tau^{-}\Tau(M,P).\]
\begin{lemma}
\label{1}\cite{AIR}
$(M,P)$ is a support tilting pair if and only if $\Tau(M,P)$ is a support tilting pair. In particular $\Tau$ and $\Tau^{-}$ induces a graph automorphism
\[G(\stilt(\Lambda))\simeq G(\stilt(\Lambda)),\]
where $G(\stilt\Lambda)$ is the underlying graph of the Hasse quiver of $\stilt(\Lambda)$.
\end{lemma}  
\begin{example}
\label{ex1}
The following are well-known example. Let $\overrightarrow{\Delta}$  be a 2-point acyclic quiver with $\overrightarrow{\Delta}_{0}=\{1,2\}$.

$(1)$ If $\overrightarrow{\Delta}$ is not connected then $\stilt(k\overrightarrow{\Delta})$ is as follows:
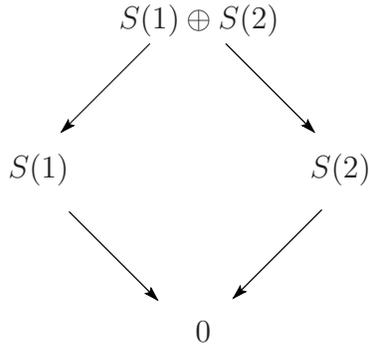
\vspace{5pt}
\[\input{kase-4th-input-example1}\]\vspace{5pt}

$(2)$ If there is a unique arrow from $1$ to $2$. Then $\stilt(k\overrightarrow{\Delta})$ is as follows:
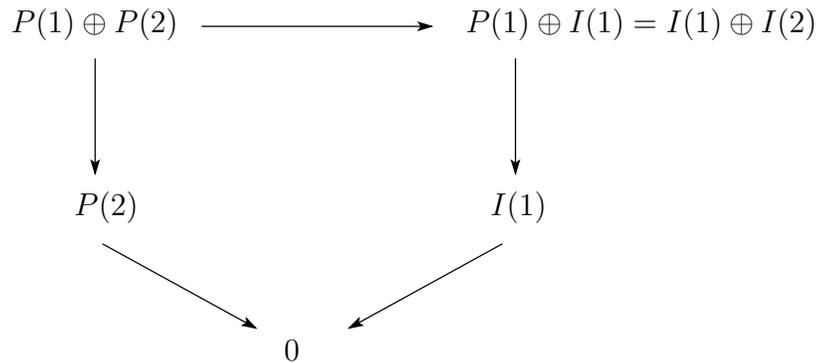
\vspace{5pt}
\[\input{kase-4th-input-example2}\]\vspace{5pt}

$(3)$ If there are at least two arrows from $1$ to $2$. Then $\stilt(k\overrightarrow{\Delta})$ is as follows:
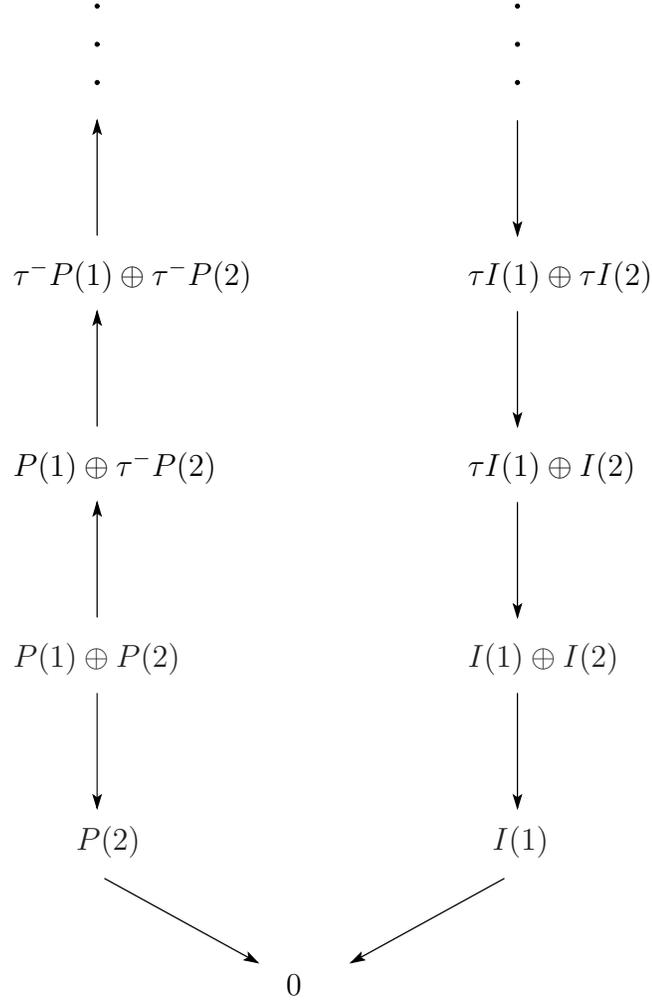
\vspace{5pt}
\[\input{kase-4th-input-example3}\]\vspace{5pt}
\end{example}

\section{Proof of Theorem\;\ref{mainthm}}
In this section, we give a proof of Theorem\;\ref{mainthm}.
Let $Q$ be a finite connected quiver, $\Lambda:=kQ$. For any $i\in \Z_{\geq 0}$, we define a full subquiver $Q(i)$ of $Q$ as follows:\\
$\bullet\;Q(0)_{0}:=\emptyset$\\
$\bullet\;Q(i)_{0}:=Q_{i-1}\cup\{a\in Q_{0}\mid a\mathrm{\ is\ a\ source\ of\ }Q\setminus Q_{i-1}\}$.

Let $m\in \Z_{\geq 1}$ be a minimum integer satisfying $Q(m)=Q$. For any $i\leq m$, we set $I_{0}:=0$ and $I_{i}:=\bigoplus_{a\in Q(i)_{0}}I(a)$. Note that we can regard $I_{i}\in \mod \Lambda$ as a basic injective tilting module of $kQ(i)$.

\subsection{Neighbours of $I_{i}$}
For a support tilting module $M$, we denote by $e(M)$ the set of direct predecessors of $M$ and $s(M)$ the set of direct successors of $M$. If $I$ is an injective support tilting module, then we have 
\[e(I)=\{I_{Q(I)\cup \{a\}}\mid a\in Q_{0}\setminus Q(I)_{0}\}\sqcup \{I/I(a)\oplus \tau_{Q(I)}S(a)\mid a\ \mathrm{is \ not\ a\ sink\ of\ }Q(I)\},\]
where $I_{Q(I)\cup\{a\}}$ be a injective tilting module of $k(Q(I)\cup \{a\})$. We also have
\[s(I)=\{I/I(a)\mid a\ \mathrm{is\ a\ sink\ of\ }Q(I)\}.\]
 It is easy to check that $I_{Q(I)\cup\{a\}}$ is injective in $\mod \Lambda$ if and only if $a$ is a source of $Q\setminus Q(I)$. In particular we obtain the following:
\[e(I_{i})=e_{1}(i)\sqcup e_{2}(i)\sqcup e_{3}(i),\;s(I_{i})=\{I_{i}/I(d)\mid d\ \mathrm{is\ a\ sink\ of}\ Q(i)\},\]
where \[\begin{array}{lll}
e_{1}(i)&:=&\{I_{i}\oplus I(a)\mid a\in Q(i+1)_{0}\setminus Q(i)_{0}\}\\
e_{2}(i)&:=&\{I_{Q_{i}\cup\{b\}}\mid b\in Q_{0}\setminus Q(i+1)_{0}\}\\
e_{3}(i)&:=&\{I_{i}/I(c)\oplus \tau_{Q(i)}S(c)\mid c\in Q(i)_{0}\mathrm{\ is\ not\ a\ sink\ of\ }Q(i)\}.\\
\end{array}
\]
\subsection{Determining injective predecessors} In this subsection we show that the set of injective predecessors $e_{1}(i)$ of $I_{i}$ is determined by poset structure of $\stilt(\Lambda)$. First we consider the case $i=0$.
\begin{lemma}\label{lemma1}
Let $a,b\in Q_{0}$. Then there is an arrow $a\rightarrow b$ in $Q$ if and only if there are $X\in e(S(a))$ and $Y\in e(S(b))$ such that
$X<Y$.
\end{lemma}
\begin{proof}
Let $\overrightarrow{\Delta}$ be a full subquiver of $Q$ with $\overrightarrow{\Delta}_{0}=\{a,b\}$. Assume that there is an arrow $a\rightarrow b$ in $Q$. We denote by $P:=P_{\overrightarrow{\Delta}}$\;(resp.\;$I:=I_{\overrightarrow{\Delta}}$)
a basic projective\;(resp.\;injective) tilting $k(\overrightarrow{\Delta})$-module. We note that $I\in e(S(a))$ and $P\in e(S(b))$ with $I<P$. 

Next we assume that there are $X\in e(S(a))$ and $Y\in e(S(b))$ such that $X<Y$. Since $Y>S(a)$ and $Y\in e(S(b))$, $Y$ must be tilting $k(\overrightarrow{\Delta})$-module with $S(b)\in \add Y$. On the other hand, $Y>X\in e(S(a))$ implies that
$X$ is $k(\overrightarrow{\Delta})$-module with $S(a)\in \add X$. If $\overrightarrow{\Delta}$ is not connected, then $X=S(a)\oplus S(b)=Y$. Therefore we have $\overrightarrow{\Delta}$ is connected. Suppose that there is an arrow $a\leftarrow b$. Then we have $\Ext_{\Lambda}^{1}(S(b),S(a))\neq 0$. Since $S(a)\in \add X$, $S(b)\in \add Y$ and $X<Y$, we reach a contradiction.
Thus there is an arrow $a\rightarrow b$. 
\end{proof}
Lemma\;\ref{lemma1} shows that $S\in e(0)$ is injective if and only if for any $S^{'},X,Y\in\stilt(\Lambda)$ with $S^{'}\in e(0)$, $X\in e(S^{'})$ and $Y\in e(S)$, we have $X\not<Y$. 
In particular, the set of injective predecessors $e_{1}(0)$ of $0$ is determined by poset-structure of $\stilt(\Lambda)$. We now assume $i>0$. 
\begin{lemma}
\label{lemma2}
Let $T\in e_{2}(i)$. 
Then there are $X,Y,Z\in \stilt(\Lambda)$ such that $X\in e(T)$, $Y\in e(I_{i})$, $Z\in e(Y)$ and $X>Z$.
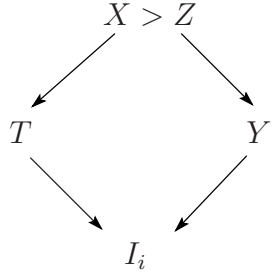
\vspace{5pt}
\[\input{kase-4th-input-lemma}\]\vspace{5pt}

\end{lemma}
\begin{proof}
By definition, there is a vertex $b\in Q_{0}\setminus Q(i+1)_{0}$ such that $T=I_{i}\oplus B=I_{Q(i)\cup\{b\}}$.
Then there is a vertex $x\in Q_{0}\setminus Q(i)$ such that $x\rightarrow b$. We note that $I_{Q(i)\cup\{x\}}$ is injective in 
$\mod k(Q(i)\cup\{x,b\})$. We also note that there is an indecomposable module $X^{'}\in \mod \Lambda$ such that $T\oplus X^{'}$ is a tilting $k(Q(i)\cup\{x,b\})$-module. 

We now assume that $T\oplus X^{'}=I_{Q(i)\cup\{x,b\}}$ and  
let $\overrightarrow{\Delta}$ be a full subquiver of $Q$ with $\overrightarrow{\Delta}_{0}=\{x,b\}$. 
We put $\Lambda^{'}=k(Q(i)\cup\{x,b\})$. Then the underlying graph $G(\stilt_{I_{i}}(\Lambda^{'}))$ of the Hasse quiver
 of $\stilt_{I_{i}}(\Lambda^{'})$ contains
\[\input{kase-4th-input-diamond}\]\vspace{5pt}
 By using Lemma\;\ref{1},
we have a graph isomorphism
\[G(\stilt_{I_{i}}(\Lambda^{'})\simeq G(\Tau^{-}_{\Lambda^{'}}(\stilt_{I_{i}}(\Lambda^{'})).\]
Note that $\Tau^{-}_{\Lambda^{'}}(\stilt_{I_{i}}(\Lambda^{'}))=\{T\in \stilt(\Lambda^{'})\mid Q(T)_{0}\subset \{x,b\}\}$.
In particular we obtain a graph isomorphism \[G(\stilt_{I_{i}}(\Lambda^{'}))\simeq G(\stilt(k\overrightarrow{\Delta})).\]
Since $\overrightarrow{\Delta}$ is a 2-point connected acyclic quiver, $G(\stilt(k\overrightarrow{\Delta}))$ has one of the following form (see Example\;\ref{ex1}):\vspace{5pt}
\[\input{kase-4th-input-g2point}\]\vspace{5pt}
Hence we reach a contradiction.
 
Therefore $X:=T\oplus X^{'}$, $Y:=I_{Q(i)\cup\{x\}}$ and $Z=I_{Q(i)\cup\{x,b\}}$ satisfy
desired property.   
\end{proof}
\begin{lemma}
\label{lemma3}
Let $T\in e_{1}(i)$. Then for any $X,Y,Z\in\stilt(\Lambda) $ with $X\in e(T)$, $Y\in e(I_{i})$ and $Z\in e(Y)$, we have $X\not> Z$.  
\end{lemma}
\begin{proof}
Suppose that there are $X,Y,Z\in \stilt(\Lambda)$ such that
\[\input{kase-4th-input-lemma}\]\vspace{5pt}

Let $a\in Q(i+1)_{0}\setminus Q(i)_{0}$ with $T=I_{i}\oplus I(a)$.

First we assume that $Y\in e_{3}(i)$. Then there is a vertex $c\in Q(i)_{0}$ which is not sink of $Q(i)$ such that \[Y=I_{i}/I(c)\oplus \tau_{Q(i)}S(c).\] In this case, $\tau_{Q(i)}S(c)\in \add Z$. Hence we have \[X=(I_{i}/I(c))\oplus I(a)\oplus \tau_{Q(i)\cup\{a\}}S(c).\] If $Q(Z)=Q(Y)=Q(i)$, then there is a vertex $z\in Q(i)_{0}\setminus\{c\}$
such that $I(z)\not\in \add Z$. On the other hand, $I(z)$ must be a direct summand of $X$. Since $X>Z$, we have \[\Ext_{\Lambda}^{1}(I(z),Z)=0=\Ext_{\Lambda}^{1}(Z,I(z)).\] Hence we obtain $I(Z)\in \add Z$. This is a contradiction. Thus we can assume that $Q(Z)=Q(Y)\cup\{z\}$ for some $z\in Q_{0}$. Then there is an indecomposable module $Z^{'}\in \mod \Lambda$ such that $Z=Y\oplus Z^{'}$. $X>Z$ implies \[Q(i)\cup\{z\}=Q(Z)\subset Q(X)=Q(i)\cup\{a\}.\]Therefore we obtain $z=a$. $X>Z$ also implies that \[\Ext_{\Lambda}^{1}(I(a),Z)=0=\Ext_{\Lambda}^{1}(Z,I(a)).\]
 We conclude that $I(a)\in \add Z$. In particular, we have that
\[Z=Y\oplus I(a).\] Now $T,X,Z\in \tilt (k(Q(i)\cup\{a\}))$ are complements of almost complete partial tilting $k(Q(i)\cup\{a\})$-module $(I_{i}/I(c))\oplus I(a)$. This is a contradiction.

Next we assume that $Y\in e_{2}(i)$. Then there is a vertex $b\in Q_{0}\setminus Q(i+1)_{0}$ such that $Q(Y)=Q(i)\cup\{b\}$. $X>Y$ implies that $Q(Y)\subset Q(X)$. Thus we obtain \[Q(X)=Q(i)\cup \{a,b\}=Q(T)\cup\{b\}\ \mathrm{and}\ X=I_{Q(i)\cup \{a,b\}}.\] If $Q(X)=Q(Z)$, then we have $X\leq Z$. This is a contradiction. Therefore we can suppose that $Q(Z)=Q(Y)$. Let $B$ be an indecomposable $\Lambda$-module such that $Y=I_{i}\oplus B$. Then $B$ must be a direct summand of $Z$. Therefore we have that $I_{i}\not\in \add Z$. On the other hand, $X>Z$ implies that $I_{i}\in \add Z$. We reach a contradiction.

Therefore $Y$ must be an element of $e_{1}(i)$. Then there is a vertex $a^{'}\in Q(i+1)_{0}\setminus Q(i)_{0}$ such that $Y=I_{i}\oplus I(a^{'})$. In this case, $X>Y$ implies that \[X=I_{i}\oplus I(a)\oplus I(a^{'}).\] This implies that $Y\in s(X)$. This contradict to $Y<Z<X$.     
\end{proof}
Let $T\in e(I_{i})$. For any $r\in Z_{\geq 1}$, we set \[\mathcal{F}(i,T,r):=\{((X_{k})_{k\in\{1,\cdots,r \}},(T_{k})_{k\in\{1\cdots ,r-1\}},(Y_{k})_{k\in\{1,\cdots, r-1\}})\mid (\star) \}\]
where $(\star)$ is the following:\\
\[(\star)\ \left\{\begin{array}{lll}
\bullet & X_{1}\in s(I_{i}),\;X_{k+1}\in s(X_{k})\\
\bullet & T_{k}\in e(X_{k})\setminus\{X_{k-1}\}\\
\bullet & Y_{k}\in e(T_{k})\\
\bullet & Y_{1}\geq T,\;Y_{k+1}\geq T_{k}\\
\end{array}\right.\]

We let $\mathcal{F}(i,T):=\bigsqcup_{r\geq 1}\mathcal{F}(i,T,r)$.
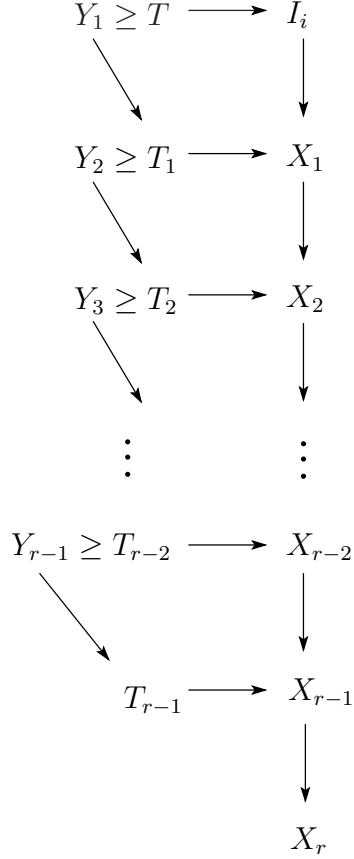
\vspace{5pt}
\[\input{kase-4th-input-flag}\]\vspace{5pt}

\begin{lemma}
\label{lemma4}
Let $T\in e_{3}(i)$. Then there exists $((X_{k}),(T_{k}),(Y_{k}))\in \mathcal{F}(i,T,r)$ such that for any $T_{r}$, $Y_{r}$ satisfying
$T_{r}\in e(X_{r})\setminus\{X_{r-1}\}$ and $Y_{r}\in e(T_{r})$, we have $ Y_{r}\not\geq T_{r-1}$
\end{lemma}
\begin{proof}
Let $T=T_{i}/I(a)\oplus \tau_{Q(i)}S(a)$. We denote by $\overrightarrow{\Delta}(i,a)$ the full subquiver of $Q$ with
\[\overrightarrow{\Delta}(i,a)_{0}:=\{x\in Q(i)_{0}\mid x\ \mathrm{is\ a\ successor\ of}\ a\}.\] We define $b_{k}\in Q(i)_{0}\;(k=1,2,\cdots,r:=\#\overrightarrow{\Delta}(i,a)_{0})$ as follows:\\
$\begin{array}{l}
(1)\;b_{1}\mathrm{\ be\ a\ sink\ of}\ \overrightarrow{\Delta}(i,a).\\
(2)\;b_{k}\mathrm{\ be\ a\ sink\ of}\ \overrightarrow{\Delta}(i,a)\setminus\{b_{1},\cdots,b_{k-1}\}.  
\end{array}$\vspace{5pt}

We set \[X_{k}:=I_{i}/\bigoplus_{j\leq k}I(b_{j}).\] We note that $X_{1}$ is a direct successor of $I_{i}$ and $X_{k+1}$ is a direct successor of $X_{k}$. For any $k\in \{1,2,\cdots, r-1\}$, we put \[Q(i,k):=Q(i)\setminus\{b_{1},\cdots,b_{k}\}\] and define $T_{k}$ and $Y_{k}$ as follows:\\
$\begin{array}{l}
\bullet\;T_{k}:=X_{k}/I_{a}\oplus\tau_{Q(i,k)}S(a).\\
\bullet\;Y_{k}\mathrm{\ is\ a\ unique\ direct\ predecessor\ of}\ T_{k}\ \mathrm{with}\ Q(Y)=Q(i,k-1).\\ 
\end{array}$\vspace{5pt}

We claim that $((X_{k}),(T_{k}),(Y_{k}))\in \mathcal{F}(i,T,r)$. If $r=1$, then the assertion is obvious. Hence we assume $r\geq 2$. Since $T_{k}\in e(X_{k})\setminus\{X_{k-1}\}, Y_{k}\in e(T_{k})$, it is sufficient to show that  
\[Y_{k+1}\geq T_{k}\;(k=0,1,\cdots,r-2),\]
where we put $T_{0}=T$. 
Without loss of generality, we can assume that $Y_{k+1}\not\simeq T_{k}$. Then poset $\stilt_{X_{k+1}/I(a)}(kQ(i,k))$ contains
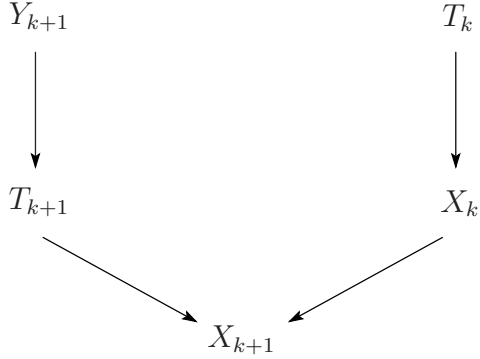
\vspace{5pt}
\[\input{kase-4th-input-lemma2}\]\vspace{5pt}

On the other hand, there is a 2-point acyclic quiver $\overrightarrow{\Delta}$ such that \[\stilt_{X_{k+1}/I(a)}(kQ(i,k))\simeq \stilt(k\overrightarrow{\Delta}).\]
If $\overrightarrow{\Delta}$ is not connected, then $\#\stilt(k\overrightarrow{\Delta})=4$ and this is a contradiction. Therefore we have that $\overrightarrow{\Delta}$ is connected. Since
$X_{k+1}$ is a minimum element of $\stilt_{X_{k+1}/I(a)}(kQ(i,k))$,
either $Y_{i+1}\geq T_{k}$ or $Y_{k+1}<T_{k}$ hold\;(see Example\;\ref{ex1}). Since $Q(Y_{k+1})=Q(i,k)=Q(X_{k})$ and $X_{k}=I_{Q(i,k)}$, we have $Y_{k+1}\geq X_{k}$. If  $Y_{k+1}<T_{k}$, 
then we conclude $X_{k}<Y_{k+1}<T_{k}$ and $T_{k}\in e(X_{k})$. We reach a contradiction. Hence we obtain
\[((X_{k}),(T_{k}),(Y_{k}))\in \mathcal{F}(i,T,r).\]

Next we suppose that there are $T_{r},Y_{r}$ such that
\[T_{r}\in e(X_{r}),Y_{r}\in e(T_{r})\ \mathrm{and}\ Y_{r}\geq T_{r-1}.\]
Note that $a$ is a sink of $Q(i)\setminus\{b_{1},\cdots,b_{r}\}=Q(X_{r})$, hence $I(a)$ must be a direct summand of $T_{r}$. Note also that $b_{r}\not\in Q(T_{r})_{0}$ and $b_{r}\in Q(T_{r-1})_{0}\subset Q(Y_{r})_{0}$. Therefore we have
$T_{r}\in \add Y_{r}.$
In particular, $I(a)$ is a direct summand of $Y_{r}$. Since $Y_{r}\geq T_{r-1}$ and $I(a)$ is an injective module, we obtain 
$I(a)\in \add T_{r-1}$. This is a contradiction.  
\end{proof}
\begin{lemma}
\label{lemma5}
Let $T\in e_{1}(i)$, $((X_{k}),(T_{k}),(Y_{k}))\in \mathcal{F}(i,T,r)$. Then there are $T_{r},Y_{r}$ such that $T_{r}\in e(X_{r})$ and $T_{r-1}\leq Y_{r}\in e(T_{r})$.
\end{lemma}
\begin{proof}
Let $T=I_{i}\oplus I(a)$. By definition, there exists $(b_{k})_{k=1,\cdots, r}\in Q(i)^{r}$ such that \[X_{k}=I_{i}/\bigoplus_{j\leq k} I(b_{j}).\]
 By using induction, we show that \[(\ast)\ \ \ Q(Y_{k})=Q(X_{k})\cup \{a,b_{k}\}\  \mathrm{and}\  Q(T_{k})=Q(X_{k})\cup \{a\}.\]

First we assume $k=1$. Then \[Q(X_{1})_{0}\subset Q(T_{1})_{0}\subset Q(Y_{1})_{0}\  \mathrm{and}\  Q(X_{1})_{0}\cup \{a,b_{1}\}=Q(T)_{0}\subset Q(Y_{1})_{0}.\] Note that $\#Q(Y_{1})_{0}-\#Q(T_{1})_{0}\leq 1$, $\#Q(T_{1})_{0}-\#Q(X_{1})_{0}\leq 1$ and $b_{1}\not\in Q(T_{1})_{0}$. Therefore we have 
\[Q(Y_{1})=Q(X_{1})\cup \{a,b_{1}\}\  \mathrm{and}\  Q(T_{1})=Q(X_{1})\cup \{a\}.\]

Next we assume $k>1$ and $(\ast)$ hold for $k-1$. Then, similar to the case $k=1$, we can check that \[Q(Y_{k})=Q(X_{k})\cup \{a,b_{k}\}\  \mathrm{and}\  Q(T_{k})=Q(X_{k})\cup \{a\}.\]   
In particular, we have $T_{k}=I_{(Q(i)\setminus\{b_{1},\cdots,b_{k} \})\cup \{a\}}$.

We now let $T_{r}:=I_{(Q(i)\setminus\{b_{1},\cdots,b_{r} \})\cup \{a\}}$
and let $Y_{r}$ be a unique direct predecessor of $T_{r}$ with $Q(Y_{r})_{0}=Q(T_{r})_{0}\cup \{b_{r}\}=Q(T_{r-1})_{0}$. Since $T_{r-1}$ is injective $k(Q(T_{r-1}))$-tilting module, we have $Y_{r}\geq T_{r-1}$.  
\end{proof}

Note that $I_{i+1}$ is a minimum element of $\bigcap_{X\in e_{1}(i)}\{T\in \stilt(\Lambda)\mid T\geq X\}$. Therefore, by combining Lemma\;\ref{lemma1}, Lemma\;\ref{lemma2}, 
Lemma\;\ref{lemma3}, Lemma\;\ref{lemma4} and Lemma\;\ref{lemma5}, we get following.
\begin{corollary}
$\tilt(\Lambda)$ is determined by poset-structure of $\stilt(\Lambda)$. In particular, if $\Lambda$ and $\Gamma$ are two finite dimensional basic hereditary algebras and $\rho$ is a poset isomorphism
from $\stilt(\Lambda)$ to $\stilt(\Gamma)$, then $\rho$ induces  a poset isomorphism \[\rho|_{\tilt(\Lambda)}:\tilt(\Lambda)\simeq \tilt(\Gamma).\]
\end{corollary}



\end{document}

%% file: kase-4th-input-example1.tex
\unitlength 0.1in
\begin{picture}( 16.4000, 16.0000)( 22.2000,-22.3000)
\put(28.0000,-8.0000){\makebox(0,0)[lb]{$S(1)\oplus S(2)$}}%
\put(22.2000,-15.8000){\makebox(0,0)[lb]{$S(1)$}}%
\put(38.0000,-15.8000){\makebox(0,0)[lb]{$S(2)$}}%
\put(32.0000,-24.0000){\makebox(0,0)[lb]{$0$}}%
%
\special{pn 8}%
\special{pa 2960 840}%
\special{pa 2500 1300}%
\special{fp}%
\special{sh 1}%
\special{pa 2500 1300}%
\special{pa 2562 1268}%
\special{pa 2538 1262}%
\special{pa 2534 1240}%
\special{pa 2500 1300}%
\special{fp}%
%
\special{pn 8}%
\special{pa 3860 1710}%
\special{pa 3400 2170}%
\special{fp}%
\special{sh 1}%
\special{pa 3400 2170}%
\special{pa 3462 2138}%
\special{pa 3438 2132}%
\special{pa 3434 2110}%
\special{pa 3400 2170}%
\special{fp}%
%
\special{pn 8}%
\special{pa 3360 840}%
\special{pa 3820 1300}%
\special{fp}%
\special{sh 1}%
\special{pa 3820 1300}%
\special{pa 3788 1240}%
\special{pa 3782 1262}%
\special{pa 3760 1268}%
\special{pa 3820 1300}%
\special{fp}%
%
\special{pn 8}%
\special{pa 2540 1720}%
\special{pa 3000 2180}%
\special{fp}%
\special{sh 1}%
\special{pa 3000 2180}%
\special{pa 2968 2120}%
\special{pa 2962 2142}%
\special{pa 2940 2148}%
\special{pa 3000 2180}%
\special{fp}%
\end{picture}%

%% file: kase-4th-input-example2
\unitlength 0.1in
\begin{picture}( 26.4500, 16.8000)( 17.0000,-25.2000)
\put(20.3000,-19.7000){\makebox(0,0)[lb]{$P(2)$}}%
\put(17.0000,-10.1000){\makebox(0,0)[lb]{$P(1)\oplus P(2)$}}%
\put(40.8000,-10.1000){\makebox(0,0)[lb]{$P(1)\oplus I(1)=I(1)\oplus I(2)$}}%
\put(42.1000,-19.7000){\makebox(0,0)[lb]{$I(1)$}}%
%
\special{pn 8}%
\special{pa 2140 1110}%
\special{pa 2140 1710}%
\special{fp}%
\special{sh 1}%
\special{pa 2140 1710}%
\special{pa 2160 1644}%
\special{pa 2140 1658}%
\special{pa 2120 1644}%
\special{pa 2140 1710}%
\special{fp}%
%
\special{pn 8}%
\special{pa 4340 1110}%
\special{pa 4340 1710}%
\special{fp}%
\special{sh 1}%
\special{pa 4340 1710}%
\special{pa 4360 1644}%
\special{pa 4340 1658}%
\special{pa 4320 1644}%
\special{pa 4340 1710}%
\special{fp}%
%
\special{pn 8}%
\special{pa 2180 2080}%
\special{pa 2980 2520}%
\special{fp}%
\special{sh 1}%
\special{pa 2980 2520}%
\special{pa 2932 2470}%
\special{pa 2934 2494}%
\special{pa 2912 2506}%
\special{pa 2980 2520}%
\special{fp}%
%
\special{pn 8}%
\special{pa 4270 2080}%
\special{pa 3470 2520}%
\special{fp}%
\special{sh 1}%
\special{pa 3470 2520}%
\special{pa 3538 2506}%
\special{pa 3518 2494}%
\special{pa 3520 2470}%
\special{pa 3470 2520}%
\special{fp}%
\put(31.3000,-26.9000){\makebox(0,0)[lb]{$0$}}%
%
\special{pn 8}%
\special{pa 2700 940}%
\special{pa 3900 940}%
\special{fp}%
\special{sh 1}%
\special{pa 3900 940}%
\special{pa 3834 920}%
\special{pa 3848 940}%
\special{pa 3834 960}%
\special{pa 3900 940}%
\special{fp}%
\end{picture}%

%% file: kase-4th-input-example3
\unitlength 0.1in
\begin{picture}( 26.4500, 50.1000)( 17.0000,-63.2000)
\put(20.3000,-57.7000){\makebox(0,0)[lb]{$P(2)$}}%
\put(17.0000,-48.1000){\makebox(0,0)[lb]{$P(1)\oplus P(2)$}}%
\put(40.8000,-48.1000){\makebox(0,0)[lb]{$I(1)\oplus I(2)$}}%
\put(42.1000,-57.7000){\makebox(0,0)[lb]{$I(1)$}}%
%
\special{pn 8}%
\special{pa 2140 4910}%
\special{pa 2140 5510}%
\special{fp}%
\special{sh 1}%
\special{pa 2140 5510}%
\special{pa 2160 5444}%
\special{pa 2140 5458}%
\special{pa 2120 5444}%
\special{pa 2140 5510}%
\special{fp}%
%
\special{pn 8}%
\special{pa 4340 4910}%
\special{pa 4340 5510}%
\special{fp}%
\special{sh 1}%
\special{pa 4340 5510}%
\special{pa 4360 5444}%
\special{pa 4340 5458}%
\special{pa 4320 5444}%
\special{pa 4340 5510}%
\special{fp}%
%
\special{pn 8}%
\special{pa 2180 5880}%
\special{pa 2980 6320}%
\special{fp}%
\special{sh 1}%
\special{pa 2980 6320}%
\special{pa 2932 6270}%
\special{pa 2934 6294}%
\special{pa 2912 6306}%
\special{pa 2980 6320}%
\special{fp}%
%
\special{pn 8}%
\special{pa 4270 5880}%
\special{pa 3470 6320}%
\special{fp}%
\special{sh 1}%
\special{pa 3470 6320}%
\special{pa 3538 6306}%
\special{pa 3518 6294}%
\special{pa 3520 6270}%
\special{pa 3470 6320}%
\special{fp}%
\put(31.3000,-64.9000){\makebox(0,0)[lb]{$0$}}%
%
\special{pn 8}%
\special{pa 2140 3510}%
\special{pa 2140 2910}%
\special{fp}%
\special{sh 1}%
\special{pa 2140 2910}%
\special{pa 2120 2978}%
\special{pa 2140 2964}%
\special{pa 2160 2978}%
\special{pa 2140 2910}%
\special{fp}%
\put(17.0000,-38.1000){\makebox(0,0)[lb]{$P(1)\oplus \tau^{-} P(2)$}}%
\put(17.0000,-28.1000){\makebox(0,0)[lb]{$\tau^{-} P(1)\oplus \tau^{-} P(2)$}}%
%
\special{pn 8}%
\special{pa 2140 4510}%
\special{pa 2140 3910}%
\special{fp}%
\special{sh 1}%
\special{pa 2140 3910}%
\special{pa 2120 3978}%
\special{pa 2140 3964}%
\special{pa 2160 3978}%
\special{pa 2140 3910}%
\special{fp}%
%
\special{pn 8}%
\special{pa 2140 2510}%
\special{pa 2140 1910}%
\special{fp}%
\special{sh 1}%
\special{pa 2140 1910}%
\special{pa 2120 1978}%
\special{pa 2140 1964}%
\special{pa 2160 1978}%
\special{pa 2140 1910}%
\special{fp}%
%
\special{pn 8}%
\special{sh 1}%
\special{ar 2140 1310 10 10 0  6.28318530717959E+0000}%
\special{sh 1}%
\special{ar 2140 1510 10 10 0  6.28318530717959E+0000}%
\special{sh 1}%
\special{ar 2140 1710 10 10 0  6.28318530717959E+0000}%
\put(40.8000,-38.1000){\makebox(0,0)[lb]{$\tau I(1)\oplus I(2)$}}%
%
\special{pn 8}%
\special{pa 4340 3910}%
\special{pa 4340 4510}%
\special{fp}%
\special{sh 1}%
\special{pa 4340 4510}%
\special{pa 4360 4444}%
\special{pa 4340 4458}%
\special{pa 4320 4444}%
\special{pa 4340 4510}%
\special{fp}%
\put(40.8000,-28.1000){\makebox(0,0)[lb]{$\tau I(1)\oplus \tau I(2)$}}%
%
\special{pn 8}%
\special{pa 4340 2910}%
\special{pa 4340 3510}%
\special{fp}%
\special{sh 1}%
\special{pa 4340 3510}%
\special{pa 4360 3444}%
\special{pa 4340 3458}%
\special{pa 4320 3444}%
\special{pa 4340 3510}%
\special{fp}%
%
\special{pn 8}%
\special{pa 4340 1910}%
\special{pa 4340 2510}%
\special{fp}%
\special{sh 1}%
\special{pa 4340 2510}%
\special{pa 4360 2444}%
\special{pa 4340 2458}%
\special{pa 4320 2444}%
\special{pa 4340 2510}%
\special{fp}%
%
\special{pn 8}%
\special{sh 1}%
\special{ar 4340 1310 10 10 0  6.28318530717959E+0000}%
\special{sh 1}%
\special{ar 4340 1510 10 10 0  6.28318530717959E+0000}%
\special{sh 1}%
\special{ar 4340 1710 10 10 0  6.28318530717959E+0000}%
\end{picture}%

%% file: kase-4th-input-lemma
\unitlength 0.1in
\begin{picture}( 12.7000, 12.7000)( 24.0000,-22.3000)
\put(30.0000,-24.0000){\makebox(0,0)[lb]{$I_{i}$}}%
\put(24.0000,-17.4000){\makebox(0,0)[lb]{$T$}}%
\put(28.9000,-11.3000){\makebox(0,0)[lb]{$X>Z$}}%
\put(36.4000,-17.4000){\makebox(0,0)[lb]{$Y$}}%
%
\special{pn 8}%
\special{pa 2510 1820}%
\special{pa 2880 2190}%
\special{fp}%
\special{sh 1}%
\special{pa 2880 2190}%
\special{pa 2848 2130}%
\special{pa 2842 2152}%
\special{pa 2820 2158}%
\special{pa 2880 2190}%
\special{fp}%
%
\special{pn 8}%
\special{pa 3640 1810}%
\special{pa 3270 2200}%
\special{fp}%
\special{sh 1}%
\special{pa 3270 2200}%
\special{pa 3330 2166}%
\special{pa 3308 2162}%
\special{pa 3302 2138}%
\special{pa 3270 2200}%
\special{fp}%
%
\special{pn 8}%
\special{pa 2950 1170}%
\special{pa 2520 1550}%
\special{fp}%
\special{sh 1}%
\special{pa 2520 1550}%
\special{pa 2584 1522}%
\special{pa 2560 1516}%
\special{pa 2558 1492}%
\special{pa 2520 1550}%
\special{fp}%
%
\special{pn 8}%
\special{pa 3300 1170}%
\special{pa 3670 1540}%
\special{fp}%
\special{sh 1}%
\special{pa 3670 1540}%
\special{pa 3638 1480}%
\special{pa 3632 1502}%
\special{pa 3610 1508}%
\special{pa 3670 1540}%
\special{fp}%
\end{picture}%

%% file: kase-4th-input-diamond.tex
\unitlength 0.1in
\begin{picture}(  8.6000,  8.6000)( 17.7000,-12.3000)
%
\special{pn 8}%
\special{ar 2200 400 30 30  0.0000000 6.2831853}%
%
\special{pn 8}%
\special{ar 1800 800 30 30  0.0000000 6.2831853}%
%
\special{pn 8}%
\special{ar 2200 1200 30 30  0.0000000 6.2831853}%
%
\special{pn 8}%
\special{ar 2600 800 30 30  0.0000000 6.2831853}%
%
\special{pn 8}%
\special{pa 2150 450}%
\special{pa 1850 750}%
\special{fp}%
%
\special{pn 8}%
\special{pa 2550 850}%
\special{pa 2250 1150}%
\special{fp}%
%
\special{pn 8}%
\special{pa 2250 450}%
\special{pa 2550 750}%
\special{fp}%
%
\special{pn 8}%
\special{pa 1850 850}%
\special{pa 2150 1150}%
\special{fp}%
\end{picture}%

%% file: kase-4th-input-g2point.tex
\unitlength 0.1in
\begin{picture}( 36.6000, 20.0000)(  7.7000,-26.0000)
%
\special{pn 8}%
\special{ar 1400 900 30 30  0.0000000 6.2831853}%
%
\special{pn 8}%
\special{ar 800 1300 30 30  0.0000000 6.2831853}%
%
\special{pn 8}%
\special{ar 800 1900 30 30  0.0000000 6.2831853}%
%
\special{pn 8}%
\special{ar 1400 2300 30 30  0.0000000 6.2831853}%
%
\special{pn 8}%
\special{ar 2000 1600 30 30  0.0000000 6.2831853}%
%
\special{pn 8}%
\special{pa 1300 970}%
\special{pa 900 1230}%
\special{fp}%
%
\special{pn 8}%
\special{pa 800 1430}%
\special{pa 800 1810}%
\special{fp}%
%
\special{pn 8}%
\special{pa 1500 970}%
\special{pa 1940 1530}%
\special{fp}%
%
\special{pn 8}%
\special{pa 1940 1690}%
\special{pa 1500 2230}%
\special{fp}%
%
\special{pn 8}%
\special{pa 900 1980}%
\special{pa 1300 2230}%
\special{fp}%
%
\special{pn 8}%
\special{ar 3600 2000 30 30  0.0000000 6.2831853}%
%
\special{pn 8}%
\special{ar 3600 1600 30 30  0.0000000 6.2831853}%
%
\special{pn 8}%
\special{ar 3600 1200 30 30  0.0000000 6.2831853}%
%
\special{pn 8}%
\special{ar 4000 1600 30 30  0.0000000 6.2831853}%
%
\special{pn 8}%
\special{ar 4400 2000 30 30  0.0000000 6.2831853}%
%
\special{pn 8}%
\special{ar 4400 1600 30 30  0.0000000 6.2831853}%
%
\special{pn 8}%
\special{ar 4400 1200 30 30  0.0000000 6.2831853}%
%
\special{pn 8}%
\special{sh 1}%
\special{ar 4400 1000 10 10 0  6.28318530717959E+0000}%
\special{sh 1}%
\special{ar 4400 800 10 10 0  6.28318530717959E+0000}%
\special{sh 1}%
\special{ar 4400 600 10 10 0  6.28318530717959E+0000}%
\special{sh 1}%
\special{ar 3600 2200 10 10 0  6.28318530717959E+0000}%
\special{sh 1}%
\special{ar 3600 2400 10 10 0  6.28318530717959E+0000}%
\special{sh 1}%
\special{ar 3600 2600 10 10 0  6.28318530717959E+0000}%
%
\special{pn 8}%
\special{pa 3670 1270}%
\special{pa 3930 1530}%
\special{fp}%
%
\special{pn 8}%
\special{pa 4070 1680}%
\special{pa 4330 1940}%
\special{fp}%
%
\special{pn 8}%
\special{pa 3600 1280}%
\special{pa 3600 1530}%
\special{fp}%
%
\special{pn 8}%
\special{pa 3600 1680}%
\special{pa 3600 1930}%
\special{fp}%
%
\special{pn 8}%
\special{pa 4400 1680}%
\special{pa 4400 1930}%
\special{fp}%
%
\special{pn 8}%
\special{pa 4400 1280}%
\special{pa 4400 1530}%
\special{fp}%
\put(26.8000,-16.8000){\makebox(0,0)[lb]{$\mathrm{or}$}}%
\end{picture}%

%% file: kase-4th-input-flag.tex
\unitlength 0.1in
\begin{picture}( 15.4500, 43.3000)( 18.7000,-47.5000)
\put(22.0000,-6.0000){\makebox(0,0)[lb]{$Y_{1}\geq T$}}%
%
\special{pn 8}%
\special{pa 2800 500}%
\special{pa 3200 500}%
\special{fp}%
\special{sh 1}%
\special{pa 3200 500}%
\special{pa 3134 480}%
\special{pa 3148 500}%
\special{pa 3134 520}%
\special{pa 3200 500}%
\special{fp}%
\put(33.1000,-5.9000){\makebox(0,0)[lb]{$I_{i}$}}%
%
\special{pn 8}%
\special{pa 3400 650}%
\special{pa 3400 1050}%
\special{fp}%
\special{sh 1}%
\special{pa 3400 1050}%
\special{pa 3420 984}%
\special{pa 3400 998}%
\special{pa 3380 984}%
\special{pa 3400 1050}%
\special{fp}%
\put(22.0000,-13.5000){\makebox(0,0)[lb]{$Y_{2}\geq T_{1}$}}%
%
\special{pn 8}%
\special{pa 2800 1250}%
\special{pa 3200 1250}%
\special{fp}%
\special{sh 1}%
\special{pa 3200 1250}%
\special{pa 3134 1230}%
\special{pa 3148 1250}%
\special{pa 3134 1270}%
\special{pa 3200 1250}%
\special{fp}%
\put(33.1000,-13.4000){\makebox(0,0)[lb]{$X_{1}$}}%
%
\special{pn 8}%
\special{pa 3400 1400}%
\special{pa 3400 1800}%
\special{fp}%
\special{sh 1}%
\special{pa 3400 1800}%
\special{pa 3420 1734}%
\special{pa 3400 1748}%
\special{pa 3380 1734}%
\special{pa 3400 1800}%
\special{fp}%
\put(22.0000,-20.9000){\makebox(0,0)[lb]{$Y_{3}\geq T_{2}$}}%
%
\special{pn 8}%
\special{pa 2800 1990}%
\special{pa 3200 1990}%
\special{fp}%
\special{sh 1}%
\special{pa 3200 1990}%
\special{pa 3134 1970}%
\special{pa 3148 1990}%
\special{pa 3134 2010}%
\special{pa 3200 1990}%
\special{fp}%
\put(33.1000,-20.8000){\makebox(0,0)[lb]{$X_{2}$}}%
%
\special{pn 8}%
\special{pa 3400 2140}%
\special{pa 3400 2540}%
\special{fp}%
\special{sh 1}%
\special{pa 3400 2540}%
\special{pa 3420 2474}%
\special{pa 3400 2488}%
\special{pa 3380 2474}%
\special{pa 3400 2540}%
\special{fp}%
%
\special{pn 8}%
\special{sh 1}%
\special{ar 3400 2770 10 10 0  6.28318530717959E+0000}%
\special{sh 1}%
\special{ar 3400 2850 10 10 0  6.28318530717959E+0000}%
\special{sh 1}%
\special{ar 3400 2940 10 10 0  6.28318530717959E+0000}%
%
\special{pn 8}%
\special{pa 2800 4060}%
\special{pa 3200 4060}%
\special{fp}%
\special{sh 1}%
\special{pa 3200 4060}%
\special{pa 3134 4040}%
\special{pa 3148 4060}%
\special{pa 3134 4080}%
\special{pa 3200 4060}%
\special{fp}%
\put(33.2000,-41.5000){\makebox(0,0)[lb]{$X_{r-1}$}}%
%
\special{pn 8}%
\special{pa 3410 4240}%
\special{pa 3410 4640}%
\special{fp}%
\special{sh 1}%
\special{pa 3410 4640}%
\special{pa 3430 4574}%
\special{pa 3410 4588}%
\special{pa 3390 4574}%
\special{pa 3410 4640}%
\special{fp}%
\put(33.3000,-49.2000){\makebox(0,0)[lb]{$X_{r}$}}%
%
\special{pn 8}%
\special{pa 2300 650}%
\special{pa 2550 1070}%
\special{fp}%
\special{sh 1}%
\special{pa 2550 1070}%
\special{pa 2534 1002}%
\special{pa 2524 1024}%
\special{pa 2500 1024}%
\special{pa 2550 1070}%
\special{fp}%
%
\special{pn 8}%
\special{pa 2300 1400}%
\special{pa 2550 1820}%
\special{fp}%
\special{sh 1}%
\special{pa 2550 1820}%
\special{pa 2534 1752}%
\special{pa 2524 1774}%
\special{pa 2500 1774}%
\special{pa 2550 1820}%
\special{fp}%
\put(24.6000,-41.9000){\makebox(0,0)[lb]{$T_{r-1}$}}%
%
\special{pn 8}%
\special{pa 2300 2130}%
\special{pa 2550 2550}%
\special{fp}%
\special{sh 1}%
\special{pa 2550 2550}%
\special{pa 2534 2482}%
\special{pa 2524 2504}%
\special{pa 2500 2504}%
\special{pa 2550 2550}%
\special{fp}%
\put(18.7000,-33.9000){\makebox(0,0)[lb]{$Y_{r-1}\geq T_{r-2}$}}%
%
\special{pn 8}%
\special{pa 2800 3300}%
\special{pa 3200 3300}%
\special{fp}%
\special{sh 1}%
\special{pa 3200 3300}%
\special{pa 3134 3280}%
\special{pa 3148 3300}%
\special{pa 3134 3320}%
\special{pa 3200 3300}%
\special{fp}%
\put(33.1000,-33.9000){\makebox(0,0)[lb]{$X_{r-2}$}}%
%
\special{pn 8}%
\special{pa 3400 3450}%
\special{pa 3400 3850}%
\special{fp}%
\special{sh 1}%
\special{pa 3400 3850}%
\special{pa 3420 3784}%
\special{pa 3400 3798}%
\special{pa 3380 3784}%
\special{pa 3400 3850}%
\special{fp}%
%
\special{pn 8}%
\special{sh 1}%
\special{ar 2480 2760 10 10 0  6.28318530717959E+0000}%
\special{sh 1}%
\special{ar 2480 2840 10 10 0  6.28318530717959E+0000}%
\special{sh 1}%
\special{ar 2480 2930 10 10 0  6.28318530717959E+0000}%
%
\special{pn 8}%
\special{pa 2020 3450}%
\special{pa 2380 3900}%
\special{fp}%
\special{sh 1}%
\special{pa 2380 3900}%
\special{pa 2354 3836}%
\special{pa 2348 3858}%
\special{pa 2324 3860}%
\special{pa 2380 3900}%
\special{fp}%
\end{picture}%

%% file: kase-4th-input-lemma2.tex
\unitlength 0.1in
\begin{picture}( 23.4500, 16.9000)( 20.0000,-25.2000)
\put(30.4000,-26.9000){\makebox(0,0)[lb]{$X_{k+1}$}}%
\put(20.0000,-19.7000){\makebox(0,0)[lb]{$T_{k+1}$}}%
\put(20.0000,-10.0000){\makebox(0,0)[lb]{$Y_{k+1}$}}%
\put(42.7000,-10.0000){\makebox(0,0)[lb]{$T_{k}$}}%
\put(42.6000,-19.7000){\makebox(0,0)[lb]{$X_{k}$}}%
%
\special{pn 8}%
\special{pa 2140 1110}%
\special{pa 2140 1710}%
\special{fp}%
\special{sh 1}%
\special{pa 2140 1710}%
\special{pa 2160 1644}%
\special{pa 2140 1658}%
\special{pa 2120 1644}%
\special{pa 2140 1710}%
\special{fp}%
%
\special{pn 8}%
\special{pa 4340 1110}%
\special{pa 4340 1710}%
\special{fp}%
\special{sh 1}%
\special{pa 4340 1710}%
\special{pa 4360 1644}%
\special{pa 4340 1658}%
\special{pa 4320 1644}%
\special{pa 4340 1710}%
\special{fp}%
%
\special{pn 8}%
\special{pa 2180 2080}%
\special{pa 2980 2520}%
\special{fp}%
\special{sh 1}%
\special{pa 2980 2520}%
\special{pa 2932 2470}%
\special{pa 2934 2494}%
\special{pa 2912 2506}%
\special{pa 2980 2520}%
\special{fp}%
%
\special{pn 8}%
\special{pa 4270 2080}%
\special{pa 3470 2520}%
\special{fp}%
\special{sh 1}%
\special{pa 3470 2520}%
\special{pa 3538 2506}%
\special{pa 3518 2494}%
\special{pa 3520 2470}%
\special{pa 3470 2520}%
\special{fp}%
\end{picture}%